\documentclass[11pt,reqno]{amsart}

\scrollmode

\usepackage{times}
\usepackage{amsmath,graphicx}
\usepackage{latexsym}
\usepackage{amscd,amsthm,amssymb,bm}
\usepackage[all]{xy}
\usepackage{tikz}
\usepackage{tikz-cd}

\newtheorem{thm}{Theorem}[section]
\newtheorem{prop}[thm]{Proposition}
\newtheorem{lem}[thm]{Lemma}

\newtheorem*{quest}{Question}

\theoremstyle{definition}

\newtheorem{defn}[thm]{Definition}

\theoremstyle{remark}
\newtheorem{rem}[thm]{Remark}

\numberwithin{equation}{section}

\title{Milnor's isospectral tori and harmonic maps}
\author{Mark J.~D.~Hamilton}
\address{      Fachbereich Mathematik\\
               Universit\"at Stuttgart\\
               Pfaffenwaldring 57\\
               70569 Stuttgart\\
               Germany}
\email{mark.hamilton@math.lmu.de}
\date{\today}
\begin{document}

\begin{abstract}
A well-known question asks whether the spectrum of the Laplacian on a Riemannian manifold $(M,g)$ determines the Riemannian metric $g$ up to isometry. A similar question is whether the energy spectrum of all harmonic maps from a given Riemannian manifold $(\Sigma,h)$ to $M$ determines the Riemannian metric on the target space. We consider this question in the case of harmonic maps between flat tori. In particular, we show that the two isospectral, non-isometric $16$-dimensional flat tori found by Milnor cannot be distinguished by the energy spectrum of harmonic maps from $d$-dimensional flat tori for $d\leq 3$, but can be distinguished by certain flat tori for $d\geq 4$. This is related to a property of the Siegel theta series in degree $d$ associated to the $16$-dimensional lattices in Milnor's example. 
\end{abstract}

\maketitle

\section{Introduction}
Let $(\Sigma,h)$ and $(M,g)$ be smooth, connected Riemannian manifolds, where $\Sigma$ is closed and oriented (in the following all manifolds are assumed smooth and connected). The Dirichlet energy of a smooth map $f\colon\Sigma\rightarrow M$ is defined by
\begin{equation}\label{eqn:def energy}
E[f]=\tfrac{1}{2}\int_\Sigma|df|^2\,\mathrm{dvol}_h,
\end{equation}
where $df\colon T\Sigma\rightarrow TM$ is the differential and $|df|^2$ is the length-squared determined by the Riemannian metrics $h$ and $g$. Stationary points of the functional $E[f]$ under variations of $f$ are called harmonic maps \cite{ES}.

The energy spectrum of harmonic maps $f\colon\Sigma\rightarrow M$, for fixed Riemannian metrics on both manifolds, is the set of critical values of the energy functional:
\begin{equation*}
\{E[f]\mid f\colon\Sigma\rightarrow M\text{ harmonic}\}\subset\mathbb{R}_{\geq 0}.
\end{equation*}
The energy spectrum of harmonic maps for specific source and target spaces has been studied in \cite{AS, CCR, V}. In general the energy spectrum does not have to be a discrete subset of $\mathbb{R}_{\geq 0}$ \cite{F, LW, L}. Depending on the situation one would also like to include the multiplicities of values in the energy spectrum, i.e.~how often an energy value is attained by several (in a suitable sense) different harmonic maps.
\begin{rem}
In the case of $\Sigma=S^1$ harmonic maps $\gamma\colon S^1\rightarrow M$ are precisely the closed geodesics in $M$. The energy of a closed geodesic is related to its length $L[\gamma]$ by $E[\gamma]=\frac{1}{4\pi R}L[\gamma]^2$, where $R$ is the radius of $S^1$, because $|\dot{\gamma}|$ is constant. The set
\begin{equation*}
\{L[\gamma]\mid \gamma\colon S^1\rightarrow M\text{ closed geodesic}\}\subset\mathbb{R}_{\geq 0}
\end{equation*}
together with multiplicities is known as the length spectrum of $(M,g)$. Here the multiplicity of a length $l$ is defined as the number of free homotopy classes of closed loops $\gamma\colon S^1\rightarrow M$ that contain a closed geodesic of length $l$ (see e.g.~\cite{GM}).
\end{rem}
If we assume that a precise notion of energy spectrum for harmonic maps is given, we can ask the following question.
\begin{quest}
Let $(M,g)$ and $(M',g')$ be Riemannian manifolds which are not isometric. Does there exist a closed, oriented Riemannian manifold $(\Sigma,h)$, such that the energy spectrum of harmonic maps $\Sigma\rightarrow M$ and $\Sigma\rightarrow M'$ is different?
\end{quest}
This question is similar to the question whether {\em ''one can hear the shape of a drum''}, i.e.~whether the spectrum of eigenvalues of the Laplacian on functions on a Riemannian manifold determines the metric up to isometry \cite{K, GWW}.\footnote{In \cite{McR} the corresponding question for a spectrum of immersed totally geodesic surfaces in Riemannian manifolds has been studied.}

In the particular case where both the source and target manifold are flat tori, the notion of energy spectrum can be given a precise meaning. We consider flat tori $(T^m=\mathbb{R}^m/\Lambda_m,g_m)$ and $(T^n=\mathbb{R}^n/\Lambda_n,g_n)$ for lattices $\Lambda_m$ and $\Lambda_n$ (without loss of generality we always use the Riemannian metric induced from the standard Euclidean scalar product on $\mathbb{R}^m$ and $\mathbb{R}^n$). It is well-known that a map $f\colon T^m\rightarrow T^n$ is harmonic if and only if it is affine (see Section \ref{sect:flat tori} for more details). Affine maps are of the form
\begin{align*}
f_{C,s}\colon T^m&\longrightarrow T^n\\
[x]&\mapsto f([x])=[Cx+s]
\end{align*}
where $C\in\mathbb{R}^{n\times m}$ with $C\Lambda_m\subset\Lambda_n$ and $s\in\mathbb{R}^n$. Since the differential acts by multiplication with the constant matrix $C$, the energy of the harmonic map $f=f_{C,s}$ is equal to
\begin{equation}\label{eqn:energy harmonic tori}
E[f]=\tfrac{1}{2}||C||^2\mathrm{vol}_{g_m}(T^m),
\end{equation}
where $||C||^2=\mathrm{Tr}(C^tC)$ and $\mathrm{Tr}$ denotes the trace, $t$ the transpose.

The set of affine maps between the tori can be identified with
\begin{equation*}
\mathrm{Hom}_{\mathbb{Z}}(\Lambda_m,\Lambda_n)\times T^n\cong \mathbb{Z}^{n\times m}\times T^n.
\end{equation*}
There is a unique affine and hence harmonic map in each homotopy class of maps from $T^m$ to $T^n$, up to translations (a map to $T^n$ is determined up to homotopy by the map $f_*$ on first integral homology, corresponding to an integral matrix of winding numbers). The translations on $T^n$ are isometries and do not change the energy of the map. We can thus define
\begin{defn}
The energy spectrum
\begin{equation*}
\{E[f]\mid f\colon T^m\rightarrow T^n\text{ harmonic}\}
\end{equation*}
of harmonic maps between flat tori is a countable subset of $\mathbb{R}_{\geq 0}$. The multiplicity of a value $E$ in the energy spectrum is defined as the number of homotopy classes of harmonic maps $f\colon T^m\rightarrow T^n$ with $E[f]=E$. This includes the case of multiplicity $0$ if $E$ does not occur in the energy spectrum.
\end{defn}
\begin{rem}
This can be generalized to harmonic maps $f\colon \Sigma\rightarrow T^n$ from any closed, oriented Riemannian manifold $(\Sigma,h)$; cf.~Section \ref{sect:flat tori}.
\end{rem}
As an explicit example we consider the $16$-dimensional isospectral tori constructed by Milnor \cite{M}. The tori are given by 
\begin{align*}
T_{8,8}&=\mathbb{R}^{16}/\Gamma_8\oplus\Gamma_8\\
T_{16}&=\mathbb{R}^{16}/\Gamma_{16}
\end{align*}
where $\Gamma_8\oplus\Gamma_8$ and $\Gamma_{16}$ are certain integral even unimodular lattices in $\mathbb{R}^{16}$ (with the restriction of the standard Euclidean scalar product). These lattices are also denoted by $E_8\oplus E_8$ and $D_{16}^+$. The flat tori $T_{8,8}$ and $T_{16}$ are non-isometric, but the spectra of the Laplacian on functions and differential forms (with multiplicities) are the same. 

Let $(\Lambda,Q)$ be an even, positive definite, unimodular lattice of rank $m$ with integral quadratic form $Q$. The Siegel upper half space of degree $d$ is defined by
\begin{equation*}
\mathcal{H}_d=\left\{Z=X+iY\left|\, X,Y\in\mathbb{R}^{d\times d}\text{ symmetric},Y\text{ positive definite}\right.\right\}\subset\mathbb{C}^{d\times d}.
\end{equation*}
For a $d$-tuple of elements $x=(x_1,\ldots,x_d)$ in $\Lambda$ define the matrix
\begin{equation*}
Q(x)=(Q(x_i,x_j))_{i,j=1,\ldots,d}\in\mathbb{Z}^{d\times d}.
\end{equation*}
The theta series \cite{Fr, Kl} of degree $d$ associated to the lattice $\Lambda$ is
\begin{equation*}
\Theta_{\Lambda}^{(d)}(Z)=\sum_{x\in\Lambda^d}e^{\pi i \mathrm{Tr}(Q(x)Z)}
\end{equation*}
where $Z\in\mathcal{H}_d$. The theta series is a Siegel modular form of degree $d$ and weight $\frac{m}{2}$. It has a Fourier expansion
\begin{equation*}
\Theta_{\Lambda}^{(d)}(Z)=\sum_{T\in \mathcal{P}_d} r_\Lambda(T)e^{\pi i \mathrm{Tr}(TZ)},
\end{equation*}
with 
\begin{equation*}
\mathcal{P}_d=\left\{T\in \mathbb{Z}^{d\times d}\left|\, T\text{ symmetric, positive semi-definite, even}\right.\right\},
\end{equation*}
where an integral matrix is called even if all of its diagonal elements are even. The Fourier coefficients are the representation numbers
\begin{equation*}
r_\Lambda(T)=\#\left\{x\in\Lambda^d\left|\, Q(x)=T\right.\right\}.
\end{equation*}
It turns out that the representation numbers and thus the theta series of degree $d$ for the lattices $\Gamma_8\oplus\Gamma_8$ and $\Gamma_{16}$ are closely related to the energy spectrum of harmonic maps from flat tori $T^d$ to $T_{8,8}$ and $T_{16}$, respectively.
\begin{prop}\label{prop:main}
If $\Theta_{\Gamma_8\oplus\Gamma_8}^{(d)}=\Theta_{\Gamma_{16}}^{(d)}$ for an integer $d\in\mathbb{N}$, then the energy spectrum (including multiplicities) of harmonic maps from any given flat torus $T^d$ to the tori $T_{8,8}$ and $T_{16}$ are the same.
\end{prop}
It is known that 
\begin{align*}
\Theta_{\Gamma_8\oplus\Gamma_8}^{(d)}&=\Theta_{\Gamma_{16}}^{(d)},\quad d=1,2,3\\
\Theta_{\Gamma_8\oplus\Gamma_8}^{(d)}&\neq\Theta_{\Gamma_{16}}^{(d)},\quad d\geq 4.
\end{align*}
For $d=1$ this follows from the fact that the dimension of the space $M_8$ of modular forms of degree $1$ and weight $8$ is equal to $1$. The case $d=2$ was proved by Witt \cite{W} in 1941 (this is the article that Milnor referred to in \cite{M}) and Witt also claimed the result for $d=4$ and conjectured the case $d=3$ (known as ''problem of Witt''). The case $d\geq 3$ was proved independently by Igusa \cite{I} and Kneser \cite{Kn} in 1967. 

Using these results we show:
\begin{thm}\label{thm:main} Consider the energy spectrum of harmonic maps from flat tori $T^d$ to $T_{8,8}$ and $T_{16}$.
\begin{enumerate}
\item\label{item:1} For any given flat torus $T^d$ of dimension $d=1,2,3$ the energy spectrum (including multiplicities) for $T_{8,8}$ and $T_{16}$ are the same.
\item\label{item:2} In every dimension $d\geq 4$ there exists a flat torus $T^d$ and an energy $E\in\mathbb{R}_{\geq 0}$ whose multiplicities in the energy spectrum for $T_{8,8}$ and $T_{16}$ are different. 
\end{enumerate}
\end{thm}
The case $d=1$ (the length spectrum of $T_{8,8}$ and $T_{16}$ are the same) was known before and implies that $T_{8,8}$ and $T_{16}$ are isospectral for the Laplacian.

\section{Harmonic maps between flat tori}\label{sect:flat tori}
Let $(T^m=\mathbb{R}^m/\Lambda_m,g_m)$ and $(T^n=\mathbb{R}^n/\Lambda_n,g_n)$ be flat tori. The following is well-known (see \cite[p.~129]{ES}, \cite{F}, \cite{NS}).
\begin{prop}
A map $f\colon T^m\rightarrow T^n$ is harmonic if and only if it is affine.
\end{prop}
For completeness we give a proof.
\begin{proof}
The proof is almost verbatim to the the case of holomorphic maps between complex tori \cite[p.~325f]{GH}. Any harmonic map $f\colon T^m\rightarrow T^n$ lifts to a harmonic map $\tilde{f}\colon \mathbb{R}^m\rightarrow\mathbb{R}^n$ with
\begin{equation*}
\tilde{f}(x+\lambda)-\tilde{f}(x)\in\Lambda_n\quad\forall x\in\mathbb{R}^m,\lambda\in\Lambda_m.
\end{equation*}
For a given vector $\lambda\in\Lambda_m$ consider the map 
\begin{align*}
\tilde{f}_\lambda\colon \mathbb{R}^m&\longrightarrow\mathbb{R}^n\\
x&\longmapsto\tilde{f}(x+\lambda)-\tilde{f}(x).
\end{align*}
Then $\tilde{f}_\lambda$ is a smooth map with image in the lattice $\Lambda_n$, hence constant. It follows that
\begin{equation*}
\partial_{x_k}\tilde{f}(x+\lambda)=\partial_{x_k}\tilde{f}(x)\quad\forall x\in\mathbb{R}^m,\lambda\in\Lambda_m,
\end{equation*}
for each $k=1,\ldots,m$, hence $\partial_{x_k}\tilde{f}$ descends to a harmonic map $T^m\rightarrow\mathbb{R}^n$. Since $T^m$ is compact, this map has to be constant by the maximum principle, which implies that $\tilde{f}$ is affine.
\end{proof}
\begin{rem}
This result can be generalized: Let $(\Sigma,h)$ be a closed, oriented Riemannian manifold and
\begin{equation*}
a\colon \Sigma\rightarrow A(\Sigma)=H_1(\Sigma;\mathbb{R})/H'_1(\Sigma;\mathbb{Z})\cong T^{b_1(\Sigma)}
\end{equation*}
the Albanese map (the prime indicates the image of integral homology in real homology). In \cite{NS} it is shown that every harmonic map $f\colon \Sigma\rightarrow T^n$ factors through $a$:
\begin{equation*}
\begin{tikzcd}
\Sigma \ar[r, "a"]\ar[rd, "f"']& A(\Sigma)\ar[d, "\phi"]\\
& T^n
\end{tikzcd}
\end{equation*}
where $\phi$ is a uniquely determined affine map. Conversely, for every affine map $\phi\colon A(\Sigma)\rightarrow T^n$ the map $f=\phi\circ a$ is harmonic.
\end{rem}

\section{Milnor's example of two isospectral $16$-dimensional tori}
Let $L_n\subset\mathbb{Z}^n$ be the lattice defined by
\begin{equation*}
\textstyle L_n=\left\{z\in\mathbb{Z}^n\left|\, \sum_{i=1}^n z_i\text{ is even}\right.\right\}
\end{equation*}
and $\Gamma_n\subset\mathbb{R}^n$ the lattice generated by $L_n$ and the vector
\begin{equation*}
\left(\tfrac{1}{2},\ldots,\tfrac{1}{2}\right)\in\tfrac{1}{2}\mathbb{Z}^n.
\end{equation*}
We consider the $16$-dimensional lattices (see \cite{M} and the discussion in \cite{BGM, C})
\begin{align*}
\Gamma_8\oplus\Gamma_8&\subset\mathbb{R}^8\oplus\mathbb{R}^8=\mathbb{R}^{16}\\
\Gamma_{16}&\subset\mathbb{R}^{16}.
\end{align*}
With the bilinear form $Q$ induced from the standard Euclidean scalar product $\langle\,,\rangle$ on $\mathbb{R}^{16}$ both lattices are integral and even, i.e.
\begin{align*}
\langle w,z\rangle&\in\mathbb{Z}\\
|w|^2&=\langle w,w\rangle \in 2\mathbb{N}\,(w\neq 0)
\end{align*}
for all vectors $w,z$ in one of the lattices.

For a lattice $\Gamma$ equal to either $\Gamma_8\oplus\Gamma_8$ or $\Gamma_{16}$ and $T$ given by the associated flat torus
\begin{align*}
T_{8,8}&=\mathbb{R}^{16}/\Gamma_8\oplus\Gamma_8\\
T_{16}&=\mathbb{R}^{16}/\Gamma_{16}
\end{align*}
consider a harmonic map
\begin{align*}
f_{C,s}\colon T^d&\longrightarrow T\\
[x]&\mapsto f([x])=[Cx+s]
\end{align*}
from a flat torus $T^d=\mathbb{R}^d/\Lambda_d$, where $C\in\mathbb{R}^{16\times d}$ with $C\Lambda_d\subset\Gamma$ and $s\in\mathbb{R}^{16}$. Let $(b_1,\ldots,b_d)$ be a fixed basis of $\Lambda_d$ and $\gamma_k=Cb_k$ the images in $\Gamma$. Writing
\begin{align*}
b&=(b_1,\ldots,b_d)\in \mathrm{GL}(d,\mathbb{R})\\
\gamma&=(\gamma_1,\ldots,\gamma_d)\in\Gamma^d\subset \mathbb{R}^{16\times d}
\end{align*}
for the matrices of column vectors we have $\gamma=Cb$ and
\begin{equation*}
C^tC=(b^t)^{-1}\gamma^t\gamma b^{-1}=(b^t)^{-1}Q(\gamma)b^{-1}.
\end{equation*}
By equation \eqref{eqn:energy harmonic tori} the energy of the harmonic map $f=f_{C,s}$ is given by
\begin{equation*}
E[f]=\tfrac{1}{2}\mathrm{Tr}\left(Q(\gamma)(b^tb)^{-1}\right)\det(b).
\end{equation*}
Since $C$ determines $\gamma$ and vice versa, the multiplicity of harmonic maps $T^d\rightarrow T$ of energy $E$ is equal to
\begin{equation}\label{eqn:mult energy gamma}
\#\left\{\gamma\in\Gamma^d\left|\, E=\tfrac{1}{2}\mathrm{Tr}\left(Q(\gamma)(b^tb)^{-1}\right)\det(b)\right.\right\}.
\end{equation}
We deduce the following:
\begin{lem}
For any given flat torus $T^d=\mathbb{R}^d/\Lambda_d$, with basis $b$ of $\Lambda_d$, and any energy $E\in\mathbb{R}_{\geq 0}$ the multiplicity of harmonic maps $T^d\rightarrow T$ with energy $E$ is equal to
\begin{equation*}
\#\left\{\gamma\in\Gamma^d\left|\, Q(\gamma)=S, E=\tfrac{1}{2}\mathrm{Tr}\left(S(b^tb)^{-1}\right)\det(b)\right.\right\}=\sum_{S\in \mathcal{Q}(E)}r_\Gamma(S),
\end{equation*}
where
\begin{equation*}
\mathcal{Q}(E)=\left\{S\in\mathcal{P}_d\mid E=\tfrac{1}{2}\mathrm{Tr}(S(b^tb)^{-1})\det(b)\right\}.
\end{equation*}
\end{lem}
The set $\mathcal{Q}(E)$ depends only on the lattice $\Lambda_d$ and not on $\Gamma$, hence Proposition \ref{prop:main} follows.

\begin{proof}[Proof of Theorem \ref{thm:main}]
Claim (\ref{item:1}) follows from Proposition \ref{prop:main} and the results of Witt, Igusa and Kneser for $d\leq 3$ mentioned above.

For claim (\ref{item:2}) it is known \cite[p.~325]{W}, \cite[p.~854]{I}, \cite{DGC} that in the case $d=4$
\begin{equation}\label{eqn:r different}
r_{\Gamma_8\oplus \Gamma_8}(S)\neq r_{\Gamma_{16}}(S)
\end{equation}
for the diagonal matrix
\begin{equation*}
S=\mathrm{diag}(2,2,2,2).
\end{equation*}
This then also holds for all $d>4$ for 
\begin{equation*}
S=\mathrm{diag}(2,2,2,2,0,\ldots,0).
\end{equation*}
We consider the case $d=4$. Let
\begin{equation*}
\arraycolsep=3pt\def\arraystretch{1.5}
M=\left(\begin{array}{cccc} 1 & \frac{1}{\pi} & \frac{1}{\pi^2} & \frac{1}{\pi^3}  \\ \frac{1}{\pi} & 1 & \frac{1}{\pi^4} & \frac{1}{\pi^5} \\ \frac{1}{\pi^2} & \frac{1}{\pi^4} & 1 & \frac{1}{\pi^6} \\ \frac{1}{\pi^3} & \frac{1}{\pi^5} & \frac{1}{\pi^6}& 1\end{array}\right).
\end{equation*}
It is easy to check that $M$ is positive definite. According to the Cholesky decomposition there exists a unique upper triangular matrix $b\in \mathrm{GL}(4,\mathbb{R})$ such that
\begin{equation*}
M^{-1}=b^tb.
\end{equation*}
The matrix $b$ defines a basis $b_1,b_2,b_3,b_4$ of $\mathbb{R}^4$. Let $\Lambda_4$ be the integral lattice spanned by these vectors and $T^4=\mathbb{R}^4/\Lambda_4$ the associated flat torus.

Denote by $\Gamma$ the lattices $\Gamma_8\oplus\Gamma_8$ and $\Gamma_{16}$. For a $4$-tuple $\gamma=(\gamma_1,\gamma_2,\gamma_3,\gamma_4)\in\Gamma^4$ we have
\begin{align*}
\mathrm{Tr}(Q(\gamma)(b^tb)^{-1})&=\mathrm{Tr}(Q(\gamma)M)\\
&=|\gamma_1|^2+|\gamma_2|^2+|\gamma_3|^2+|\gamma_4|^2\\
&\quad +2\left(\tfrac{1}{\pi}\langle \gamma_1,\gamma_2\rangle+\tfrac{1}{\pi^2}\langle \gamma_1,\gamma_3\rangle+\tfrac{1}{\pi^3}\langle \gamma_1,\gamma_4\rangle\right)\\
&\quad \left.+\tfrac{1}{\pi^4}\langle \gamma_2,\gamma_3\rangle+\tfrac{1}{\pi^5}\langle \gamma_2,\gamma_4\rangle+\tfrac{1}{\pi^6}\langle \gamma_3,\gamma_4\rangle  \right).
\end{align*}
Consider the energy
\begin{equation*}
E=\tfrac{1}{2}\cdot8\cdot\det(b).
\end{equation*}
A harmonic map $f_{C,s}\colon T^4\rightarrow T$, determined by $C$ or equivalently $\gamma$, has energy $E$ if any only if
\begin{equation*}
8=\mathrm{Tr}(Q(\gamma)(b^tb)^{-1}).
\end{equation*}
For the chosen basis this is happens if and only if
\begin{align}
|\gamma_1|^2+|\gamma_2|^2+|\gamma_3|^2+|\gamma_4|^2&=8\label{eqn:gamma1}\\
\langle \gamma_i,\gamma_j\rangle&= 0\quad\forall i\neq j,\label{eqn:gamma2}
\end{align}
because $\langle\,,\rangle$ is integral on $\Gamma$. The second equation means that $Q(\gamma)$ has to be diagonal. If all lattice vectors $\gamma_i$ are non-zero, equation \eqref{eqn:gamma1} is equivalent to
\begin{equation}\label{eqn:gamma3}
|\gamma_1|^2=|\gamma_2|^2=|\gamma_3|^2=|\gamma_4|^2=2,
\end{equation}
because the lattice is even and positive definite. By the result mentioned at the beginning of the proof, the number of $4$-tuples $(\gamma_1,\gamma_2,\gamma_3,\gamma_4)$ that satisfy equation \eqref{eqn:gamma3} and \eqref{eqn:gamma2} is different for $\Gamma_8\oplus\Gamma_8$ and $\Gamma_{16}$.

If some of the $\gamma_i$ are zero, there are several possibilities, for example,
\begin{equation}
|\gamma_1|^2=4,\,|\gamma_2|^2=|\gamma_3|^2=2,\,|\gamma_4|^2=0.
\end{equation}
However, the number of $4$-tuples of this type is the same for both lattices, because $\Theta_{\Gamma_8\oplus\Gamma_8}^{(d)}=\Theta_{\Gamma_{16}}^{(d)}$ for $d\leq 3$. 

It follows that the total number of $4$-tuples $\gamma$ that satisfy \eqref{eqn:gamma1} and \eqref{eqn:gamma2} is different for $\Gamma_8\oplus\Gamma_8$ and $\Gamma_{16}$. This proves the claim with \eqref{eqn:mult energy gamma}.
\end{proof}

\bibliographystyle{amsplain}

\bigskip
\bigskip

\end{document}